\documentclass[12pt]{article}
\usepackage{amsmath, amsfonts,amsthm,amssymb,amsbsy,upref,color,graphicx,amscd,hyperref,makeidx,enumerate}
\usepackage{tikz}
\usepackage[active]{srcltx}
\usepackage[latin1]{inputenc}
\usepackage{enumerate}
\usepackage[english]{babel}
\usepackage{newlfont}
\usepackage{amsbsy}
\usepackage[english]{babel}
\usepackage[center]{caption2}
\usepackage{amssymb}
\usepackage{amsmath}
\usepackage{latexsym}
\usepackage{amsthm}
\usepackage{amsthm}
\theoremstyle{plain}
\newtheorem{thm}{Theorem}

\newtheorem{prop}[thm]{Proposition}
\newtheorem{nota}[thm]{Notation}
\newtheorem{rem}[thm]{Remark}
\newtheorem{defin}[thm]{Definition}

\newcommand{\R}{\mathbb{R}}

\newcommand{\N}{\mathbb{N}}
\newcommand{\C}{\mathbb{C}}

\usepackage{mathtools}
\def\multiset#1#2{\ensuremath{\left(\kern-.2em\left(\genfrac{}{}{0pt}{}{#1}{#2}\right)\kern-.2em\right)}}
\usepackage[center]{caption2}

\begin{document}

\title{A characterization of dissimilarity families  of trees}
\author{Agnese Baldisserri, Elena Rubei}
\date{}
\maketitle

\def\thefootnote{}
\footnotetext{ \hspace*{-0.56cm}

 }

\begin{abstract}
Let  ${\cal T}=(T,w)$ be a weighted  finite tree with leaves $1,..., n$.
For any  $I :=\{i_1,..., i_k \} \subset \{1,...,n\}$,
let $D_I ({\cal T})$ be
the weight of the minimal subtree of $T$ connecting $i_1,..., i_k$;
the  $D_{I} ({\cal T})$ are called 
$k$-weights of  ${\cal T}$. Given  a family 
of  real numbers parametrized by the $k$-subsets of $
\{1,..., n\}$, $\{D_I\}_{I  \in {\{1,...,n\} \choose k}}$,
  we say that a weighted tree ${\cal T}=(T,w)$ with leaves $1,..., n$ realizes the family 
if $D_I({\cal T})=D_I$ for any $ I  $.

In 2006 Levy, Yoshida and Pachter defined, for any 
  positive-weighted tree  ${\cal T}=(T,w)$ with $\{1,..., n\}$ as leaf set and any $i, j \in \{1,..., n\}$,
the numbers $S_{i,j}$ to be $ \sum_{Y \in {\{1,..., n\} -\{i,j\} \choose k-2}} D_{i,j ,Y}({\cal T})  $; they proved 
 that there exists a positive-weighted tree ${\cal T}' =(T',w')$   such that $D_{i,j}({\cal T}')=S_{i,j}$  for any $i,j \in \{1,..., n\}$ and that this new tree is, in some way, similar to the given one.
In this paper, by using the $S_{i,j}$ defined by Levy, Yoshida and Pachter,  we characterize families of real numbers parametrized by ${\{1,...,n\} \choose k}$ that are the families of $k$-weights of weighted trees with leaf set equal to $\{1,...., n\}$ and weights of the internal edges positive.
\end{abstract}

\section{Introduction}

For any graph $G$, let $E(G)$, $V(G)$ and $L(G)$ 
 be respectively the set of the edges,   
the set of the vertices and  the set of the leaves of $G$.
A {\bf weighted graph} ${\cal G}=(G,w)$ is a graph $G$ 
endowed with a function $w: E(G) \rightarrow \R$. 
For any edge $e$, the real number $w(e)$ is called the weight of the edge. If all the weights are nonnegative (respectively positive), 
we say that the graph is {\bf nonnegative-weighted} (respectively {\bf positive-weighted});
if  the weights  of 
the internal edges are nonzero,
 we say that the graph is {\bf internal-nonzero-weighted} and,
if the weights  of 
the internal edges are positive,
 we say that the graph is {\bf internal-positive-weighted}.
For any finite subgraph $G'$ of  $G$, we define  $w(G')$ to be the sum of the weights of the edges of $G'$. 
In this paper we will deal only with weighted finite trees.

\begin{defin}
Let ${\cal T}=(T,w) $ be a weighted tree. 
For any distinct $ i_1, .....,i_k \in V(T)$,
 we define $ D_{\{i_1,...., i_k\}}({\cal T}) $ to be the weight of the minimal 
subtree containing $i_1,....,i_k$. We call this subtree ``the subtree realizing  $ D_{\{i_1,...., i_k\}}({\cal T}) $''.
 More simply, we denote 
$D_{\{i_1,...., i_k\}}({\cal T})$ by
$D_{i_1,...., i_k}({\cal T})$ for any order of $i_1,..., i_k$.  
We call  the  $ D_{i_1,...., i_k}({\cal T})$ 
the {\bf $k$-weights} of ${\cal T}$ and
 we call  a $k$-weight of ${\cal T}$ for some $k$  a {\bf multiweight}
 of ${\cal T}$.
\end{defin}


 If $S $ is a subset of $V(T)$, the $k$-weights $D_{i_1,\ldots,  i_k}({\cal T})$ with $i_1, \ldots, i_k \in S$ give
a vector in $\mathbb{R}^{  S \choose k}$. This vector is called 
$k${\bf -dissimilarity vector} of $({\cal T}, S)$.
Equivalently,  we can speak 
of the {\bf family of the $k$-weights}  of $({\cal T}, S)$ or of the {\bf  $k$-dissimilarity family of $({\cal T}, S)$.}   

If $S$ is a finite set, $k \in \N$ and $k < \# S$,
we say that a family of  real numbers
 $\{D_{I}\}_{I \in {S \choose k}}$  is {\bf treelike} (respectively p-treelike, nn-treelike, 
inz-treelike, ip-treelike) if there exist a weighted  (respectively 
positive-weighted,  nonnegative-weighted, internal-nonzero-weighted, 
internal-positive-weighted)  tree
 ${\cal T}=(T,w)$ and a subset $S$ 
of the set of its vertices such that $ D_{I}({\cal T}) = D_{I}$  for any 
 $k$-subset $I$ of $ S$. 
In this case, we say also that  ${\cal T}$   realizes the family $\{D_I\}_{I  \in {S \choose k}}$. 
If in addition 
$S \subset L(T)$,
we say that the family
 is  {\bf l-treelike} (respectively p-l-treelike, nn-l-treelike, inz-l-treelike,
  ip-l-treelike).

Weighted graphs have applications in several disciplines, such as 
biology and psychology. Phylogenetic trees are weighted graphs
whose vertices represent  species and the weight of an edge is given 
by how much the DNA sequences of the species represented by the vertices of the edge differ. Dissimilarity families arise naturally also in 
 psychology, see for instance the introduction in \cite{C-S}.
There is a wide literature concerning graphlike dissimilarity families
and treelike dissimilarity families,
in particular concerning methods to reconstruct weighted trees from their dissimilarity families; these methods, for instance the so-called neighbor-joining method, are used by biologists to reconstruct phylogenetic trees. 
See  for example \cite{N-S}, \cite{S-K} and   \cite{Dresslibro}, \cite{S-S} for overviews. We recall the most important results concerning  treelike dissimilarity families.

A criterion for a metric on a finite set to be nn-l-treelike
was established in  \cite{B}, \cite{SimP}, \cite{Za}: 

\begin{defin}
Let $\{D_{I}\}_{I \in {\{1,...,n\} \choose 2}}$ be a family of positive real numbers. We say that the $D_I$  satisfy  the {\bf  $4$-point condition} if and only if 
for all distinct $a,b,c,d  \in \{1,...,n\}$,
the maximum of $$\{D_{a,b} + D_{c,d},D_{a,c} + D_{b,d},D_{a,d} + D_{b,c}
 \}$$ is attained at least twice. 
\end{defin}

\begin{thm} \label{Bune} 
Let $\{D_{I}\}_{I \in {\{1,...,n\} \choose 2}}$ be a family of positive real numbers satisfying the triangle inequalities.
It is p-treelike (or nn-l-treelike) 
 if and only if the  $4$-point condition holds.
\end{thm}

Also the study of general weighted trees can be interesting and,   
in \cite{B-S}, Bandelt and Steel proved a result, analogous to
Theorem \ref{Bune}, for general weighted trees:

\begin{thm} \label{B-S} {\bf (Bandelt-Steel)}
 For any family of real numbers $\{D_{I}\}_{I \in {\{1,...,n\} \choose 2}}$,
 there exists a weighted tree ${\cal T}$ with leaves $1,...,n$
such that $ D_{I} ({\cal T})= D_{I}$  for any $I \in {\{1,...,n\} \choose 2}$  if and only  if the so-called relaxed $4$-point condition holds, i.e. for any $a,b,c,d \in  \{1,...,n\}$, at least two among 
 $ D_{a,b} + D_{c,d},\;\;D_{a,c} + D_{b,d},\;\; D_{a,d} + D_{b,c}$
are equal.
\end{thm}

An easy variant of the theorems above is the following:

\begin{thm} \label{Bune2} 
 For any family of real numbers $\{D_{I}\}_{I \in {\{1,...,n\} \choose 2}}$,
 there exists an internal-positive weighted tree ${\cal T}$ with leaves $1,...,n$
such that $ D_{I} ({\cal T})= D_{I}$  for any $I \in {\{1,...,n\} \choose 2}$  if and only  if the $4$-point condition holds.
\end{thm}

In fact, if the $4$-point condition holds, in particular the relaxed $4$-point condition holds, so by Theorem \ref{B-S}, there exists a weighted tree ${\cal T}$  with leaves $1,...,n$ and with $2$-weights equal to the $D_I$; it is easy to see that, since the $4$-point condition holds, the weights of the internal edges of  ${\cal T}$ are nonnegative; by contracting the edges of weight $0$, we get an ip-weighted tree  with leaves $1,...,n$ and with $2$-weights equal to the $D_I$.

For higher $k$ the literature is more recent, see  \cite{B-R},  \cite{BC},
 \cite{H-H-M-S},   \cite{Iri}, \cite{L-Y-P},  \cite{Man}, \cite{P-S}, \cite{Ru1}, \cite{Ru2}.
Three of the most important results for higher $k$ are the following:

\begin{thm} {\bf  (Herrmann, Huber, Moulton, Spillner, \cite{H-H-M-S})}.
If $n \geq 2k$, a family  of positive real numbers 
$ \{D_{I}\}_{I \in {\{1,..., n\} \choose k}}$ is ip-l-treelike 
if and only if its restriction to every $2k$-subset of $\{1,...,n\}$ is 
ip-l-treelike.   
\end{thm}

\begin{thm} \label{PatSp} {\bf (Pachter-Speyer, \cite{P-S})}. Let $ k ,n  \in \mathbb{N}$ with
$3 \leq  k \leq (n+1)/2$.  A positive-weighted tree
 ${\cal T}$ with leaves $1,...,n$ and no vertices of degree 2
is determined by the values $D_I({\cal T})$, where $ I $ varies in  
${\{1,...,n\} \choose k }$.
\end{thm}

 \begin{thm} \label{LYP} {\bf (Levy-Yoshida-Pachter,  \cite{L-Y-P})} Let ${\cal T}=(T,w)$ be a positive-weighted tree 
  with $L(T)=\{1,..., n\}$. For any distinct $i ,j \in \{1,..., n\}$, define $$S_{i,j} = \sum_{Y \in {\{1,..., n\} -\{i,j\} \choose k-2}} D_{i,j ,Y} ({\cal T}). $$ Then there exists a positive-weighted tree ${\cal T}' =(T',w')$   such that 
$D_{i,j}({\cal T}')=S_{i,j}$  for all distinct  $i,j \in \{1,..., n\}$,
  the quartet system of $T'$ is contained in  the quartet system of $T$ and, defined $T_{\leq s}$ the subforest of $T$  whose edge set consists of edges whose removal results in one of the components having size at most $s$, we have 
  $T_{\leq n-k} \cong T'_{\leq n-k}$. 
  \end{thm}

Moreover Levy, Yoshida and  Pachter proposed a neighbor-joining  algorithm 
for reconstructing trees from $k$-weights.
To prove the first statement of Theorem \ref{LYP}, Levy, Yoshida and  Pachter proved 
that the $S_{i,j}$ satisfy the  $4$-point condition. 
It is natural to wonder if the  $4$-point condition for the $S_{i,j}$ and some other possible conditions could be sufficient 
for a family $ \{D_{I}\}_{I \in {\{1,..., n\} \choose k}}$ to be l-treelike. An easy   argument about the numbers of the $k$-weights, the numbers of 
the equations given by the $4$-point condition and the numbers of 
edges of a tree with $n$ leaves 
 suggests that the  $4$-point condition for the $S_{i,j}$ cannot be suffficient to characterize  l-treelike families.
In this paper, by using the $S_{i,j}$ defined by Levy, Yoshida and Pachter,  we give a characterization of families of real numbers parametrized by ${\{1,...,n\} \choose k}$ that are the families of $k$-weights of ip-weighted trees with leaf set equal to $\{1,...., n\}$
(see Theorem \ref{ip-l}).


\section{Notation and some remarks}

\begin{nota} \label{notainiziali}


$\bullet$  For any $n \in \N $ with $ n \geq 1$, let $[n]= \{1,..., n\}$.

$\bullet$  For any set $S$ and $k \in \mathbb{N}$,  let ${S \choose k}$
be the set of the $k$-subsets of $S$.

$\bullet $ Let $S$ be a set and $f:S \rightarrow \R$ be a function.
For any $A, B $ subsets of $S$ and any $a,b \in \R$,  we denote $a \sum_{x \in A} f(x) + b \sum_{x \in B} f(x)$ by $$\left( a\, \sum_{x \in A} + \;b \, \sum_{x \in B} \right) f(x).$$

$\bullet $ For any set $S$ and any $i \in S$ and $X \subset S$, we write $iX$ instead of $\{i\} \cup X$. 

$\bullet$ Throughout the paper, the word ``tree'' will denote a finite  tree.

$\bullet$ A {\bf node} of a tree is a vertex of degree greater than 2.

$\bullet$  Let $F$ be a leaf of a tree $T$. Let $N$ be the node 
 such that the path $p$ between $N$ and $F$ does not contain any node apart from $N$. We say that $p$ is the {\bf twig} associated to $F$.
We say that an edge is {\bf internal} if it is not an edge of a twig.
We denote by $\mathring{E}(T)$ the set of the internal edges of $T$.

$\bullet $ We say that a tree is {\bf essential} if it has no vertices of degree $2$. 

$\bullet $ If $a$ and $b$ are vertices of a tree, we denote by $p(a,b)$ the path between $a$ and $b$.

$\bullet$ Let $T$ be a tree and let $S$ be a subset of $L(T)$. We denote by $T|_S$ the minimal subtree 
of $T$ whose set of vertices  contains $S$. If ${\cal T}= (T,w)$ is a weighted tree, we denote by 
${\cal T}|_S$  the tree $T|_S$ with the weighting induced by $w$.

$\bullet$ Let $T$ be a tree, $T'$ be a subtree of $T$ and $S$ be a subtree of $T'$. Let $x \in L(T) - L(T')$. 
We say that   $x$ {\bf clings to} $S$ as to $T'$ if the minimal subtree of $T$ containing
$S$ and $x$  has no edges in common with the complementary of $S$ in $T'$. See Figure \ref{clings} for an example: let $T$ be the tree in the 
figure and let $T'=T|_{a,b,c,d}$ and  $S=p(a,b)$.
\end{nota}

\begin{figure}[h!]
\begin{center}

\begin{tikzpicture} \label{clings}
\draw [thick] (0,0) --(1,0);
\draw [thick] (0,0) --(-1,0);
\draw [thick] (1,0) --(2,1);
\draw [thick] (1,0) --(2,-1);
\draw [thick] (-1,0) --(-2,1);
\draw [thick] (-1,0) --(-2,-1);
\draw [thick] (-1.5,-0.5) --(-2.5,-0.8);
\draw [thick] (1.3,0.3) --(1.4,1);
\draw [thick] (-0.5,-0.8) --(-0.2,0);
\draw [thick] (-0.1,-0.8) --(-0.3,-0.3);
\node[above] at (-2, 1) {a};
\node[below] at (-2,-1 ) {b};
\node[above] at (2, 1) {c};
\node[below] at (2,-1 ) {d};
\node[left] at (-2.5,-0.8) {x};
\end{tikzpicture}

\caption{ $x$ clings  to $S:=p(a,b)$ as to $T':=T|_{a,b,c,d}$}
\end{center}
\end{figure}
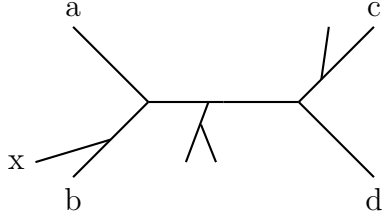

\begin{defin} \label{cherries} Let $T$ be a tree.

We say that two leaves  $i$ and $j$ of $T$ are {\bf neighbours}
if in $p(i,j)$ there is only one node; 
furthermore, we say that  $C \subset L(T)$ is a {\bf cherry} if any $i,j \in C$ are neighbours.


The {\bf stalk} of a cherry is the unique node in the path 
with endpoints any two elements of the cherry.

Let  $a,b,c,d \in L(T)$. We say that $ \langle
a, b | c, d \rangle $ holds if  in  $T|_{\{a,b,c,d\}}$ we have that $a$ and $b$ are neighbours, 
 $c$ and $d$ are neighbours, and  $a$ and $c$ are not neighbours; in this case we denote by  $\gamma_{a,b,c,d}$ the path 
between the stalk $s_{a,b}$ of $\{a,b\}$ and the stalk  $s_{c,d}$ of $ \{c,d\}$ in $T|_{\{a,b,c,d\}}$; we call it the {\bf bridge} of the quartet $(a,b,c,d)$. The symbol 
$ \langle a,b \,| \, c, d \rangle $  is called {\bf Buneman's index} of $a,b,c,d$.

\end{defin}

\begin{defin} Let $k\in \N- \{0\}$.  We say that  a tree $P$ is a {\bf pseudostar} of kind $(n,k)$ if $\# L(P)=n$ and
   any edge  of $P$ divides $L(P)$
 into two sets such that  at least one of them has cardinality   greater than or equal to $ k$.
\end{defin}

\begin{figure}[h!]
\begin{center}

\begin{tikzpicture}
\draw [thick] (0,0) --(0.2,1);
\draw [thick] (0,0) --(1.1,0.7);
\draw [thick] (0,0) --(1.5,0);
\draw [thick] (0,0) --(1.4,-0.4);
\draw [thick] (0,0) --(-1.8,-0.4);
\draw [thick] (-1.8,-0.4) --(-2.5,-0.3);
\draw [thick] (-1.8,-0.4) --(-2.2,-1);
\draw [thick] (0,0) --(-1.3,-1.3);
\draw [thick] (-1.3,-1.3) --(-2,-1.4);
\draw [thick] (-1.3,-1.3) --(-1.5,-2);
\draw [thick] (0,0) --(-0.3,-1.8);
\draw [thick] (-0.3,-1.8) --(-0.8,-2.4);
\draw [thick] (-0.3,-1.8) --(0.1,-2.4);
\end{tikzpicture}

\caption{ A pseudostar of kind $(10,8)$ \label{pseudostar}}
\end{center}
\end{figure}
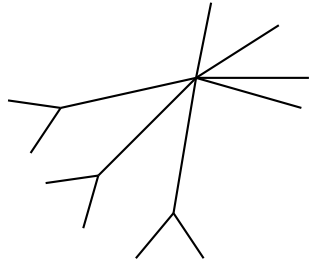

\begin{rem}
(i) A pseudostar of kind $(n,n-1) $ is a star, that is, a tree with only one node.

(ii) Let $ k,n \in \N-\{0\}$. If $\frac{n}{2} \geq k $, then every tree 
with $n$ leaves is a pseudostar of kind $(n,k)$, in fact, if we divide a set with $n$ elements into two 
parts, at least one of them has cardinality greater than or equal to $ \frac{n}{2}$, which 
is greater than or equal to $ k$. 
\end{rem}

\begin{thm} \label{main} {\bf (Baldisserri-Rubei, \cite{B-R2})} Let $n, k \in \N$ with $ 3 \leq k \leq n-1$.
Let $\{D_I\}_{I \in {[n] \choose k}}$ be a  family of  real numbers.
If it is l-treelike, then there exists exactly  one internal-nonzero-weighted
essential pseudostar ${\cal P} $ of kind $(n,k)$ realizing the family.

If the family $\{D_I\}_{I \in {[n] \choose k}}$ is  
p-l-treelike, then  ${\cal P} $ is positive-weighted.
\end{thm}


\section{Characterization of treelike families}

\begin{defin} \label{Sij}
Let $\{D_{I}\}_{I \in {[n] \choose k}}$ be a family of  real numbers.
For any distinct $i ,j \in [n]$, define $$S_{i,j} = \sum_{Y \in {[n] -\{i,j\} \choose k-2}} D_{i,j ,Y}  $$ 
\end{defin}

\begin{defin} \label{defL} Let $\{S_{i,j}\}$ for $i, j \in [n]$ with $i \neq j$ be a family
of real numbers.  For any distinct $a,b,c,d \in [n]$, 
let 
$$ L^{\{a,b,c,d\}}_{\{a,b\}} =\left\{x \in [n]-\{a,b,c,d\} \; \bigg| \;   \begin{array}{ll} \mbox{ either } & S_{x,z} - S_{a,z}  \mbox{ does not depend on } z \in \{b,c,d\}  
\\ \mbox{ or } &   S_{x,z} - S_{b,z}  \mbox{ does not depend on } z \in 
\{a,c,d\} \end{array}
\right\}.$$
We will denote $ L^{\{a,b,c,d\}}_{\{a,b\}}$ simply by $ L^{a,b,c,d}_{a,b}$ and 
we will omit the superscript  when the $4$-set 
which we are referring to is clear from the context.
\end{defin}

\begin{prop} \label{bridge}
Let ${\cal A} =(A,w) $ be an internal-positive-weighted essential tree. Let $S_{i,j}$ for distinct $i,j \in [n]$ be the $2$-weights of ${\cal A}$. Let $a,b,c,d \in [n]$. 

1) If $\langle a,b \, |\, c,d \rangle $  holds,  we have that $L_{a,b}$ 
is the set of the elements $x$ of $[n]$ clinging to  $p(a,b)$
as to $T|_{a,b,c,d}$ and
$L_{c,d}$ is the set of the elements 
$x$ of $[n]$ clinging to  $p(c,d)$ as to $T|_{a,b,c,d}$.

2) We have that 
$\langle a,b \, |\, c,d \rangle $ holds and  the bridge of $(a,b,c,d)$ is given by exactly one   edge if and only if the following conditions hold:

(i)  $$S_{a,b} + S_{c,d} <
 S_{a,c} + S_{b,d} = S_{a,d}+ S_{b,c}$$ 

(ii) $L_{a,b} \cup L_{c,d}= [n]$.

\end{prop}

\begin{proof}
1) Observe that $L_{a,b} $ is the set of the elements $x$ of $[n]$ that 
are neighbours either  of $a$ or of $b$ in $A|_{a,b,c,d,x}$; hence, 
 if $\langle a,b \, |\, c,d \rangle $  holds,  we have that $L_{a,b}$ 
is the set of the elements $x$ of $[n]$ clinging to  $p(a,b)$
as to $A|_{a,b,c,d}$. 

2)
$\Longrightarrow $ Suppose $\langle a,b \, |\, c,d \rangle $ holds and  the bridge of $(a,b,c,d)$ is given by exactly one  edge;  then  the weight of the bridge is positive, 
so (i) holds; moreover, $L_{a,b} $ 
is the set of the elements $x$ of $[n]$ clinging to  $p(a,b)$ and
$L_{c,d} $ 
is the set of the elements $x$ of $[n]$ clinging to  $p(c,d)$. 
 So (ii) must hold.

$\Longleftarrow$ 
If $ \langle a,b \, |\, c,d \rangle $ did not hold, then either $A|_{a,b,c,d}$ would be 
a star or  one of  $\langle a,c \, |\, b,d \rangle $ and  $\langle a,d \, |\, b, c \rangle $ would  hold.  So we would have either $ S_{a,b} + S_{c,d} = S_{a,d}+ S_{b,c}$ or 
 $ S_{a,b} + S_{c,d} = S_{a,c}+ S_{b,d}$, which is absurd by assumption (i).
 Hence $\langle a,b \, |\, c,d \rangle $ holds.  Moreover,
if the bridge of $(a,b,c,d)$ were given by more than one edge, then, since $A$ is essential,  there would 
exist $x \in [n]$ clinging to the bridge, and so we would have $ x \not \in L_{a,b}
\cup L_{c,d}$, which is absurd by condition (ii). 
\end{proof}

\begin{rem} \label{wei}
Let ${\cal T}=(T,w) $ be a weighted essential tree with $L(T)=[n]$ and let  $k$ be a natural number less than $n$. Let $e_i$ denote the twig associated to $i$ for any $i \in [n]$.  Then 
$$ w(e_i) = \frac{D_I({\cal T)}}{k} -\frac{1}{k}
\sum_{e \in \mathring{E} (T|_I) } w(e) + \frac{1}{k} \sum_{j \in I} \left( 
D_{j,X} ({\cal T})- D_{i,X} ({\cal T}) - \sum_{e \in \mathring{E} (T|_{jX})}
w(e) +\sum_{e \in \mathring{E} (T|_{iX})} w(e)
\right)
$$
for any $i \in [n]$,  $ I \in {[n] \choose k}$  and $X=X_{i,j} \in {[n] -\{i,j\}\choose k-1}$ (depending on $i$ and $j$). 
\end{rem}

\begin{proof}
Let $ I \in {[n] \choose k}$. Then 
\begin{equation}  \label{wei1} D_I ({\cal T}) = \sum_{i \in I} w(e_i) + \sum_{e \in \mathring{E} (T|_I) } w(e). \end{equation}
Thus, for any $i, j \in [n]$,
\begin{equation}  \label{wei2}
w(e_j) - w(e_i) = D_{j,X} ({\cal T})- D_{i,X} ({\cal T}) - \sum_{e \in \mathring{E} (T|_{jX})}
w(e) +\sum_{e \in \mathring{E} (T|_{iX})} w(e) \end{equation}
for any $X   \in {[n] \choose k-1}$ such that $i, j \not \in X$.
Obviously, for any $i \in [n]$ and any $ I \in {[n] \choose k}$, we have:
\begin{equation}  \label{wei3}
k \, w(e_i) = \sum_{j \in I} \left( w(e_i ) - w(e_j) \right)
+ \sum_{j \in I} w(e_j). \end{equation} 
From (\ref{wei1}), (\ref{wei2}) and (\ref{wei3}), we get easily our assertion. 
\end{proof}

\begin{prop} \label{T'pseudo} Let ${\cal T}$ be an ip-weighted tree with $L(T)=[n]$. 
Let $S_{i,j}$ for distinct  $i,j \in [n]$ be defined from 
the $ D_I ({\cal T})$ for  $I \in {[n ]\choose k}$
as in Definition \ref{Sij}.
Let ${\cal T}'$ be an ip-weighted  tree with $2$-weights the $S_{i,j}$ (the existence follows from Theorem \ref{Bune2} and paper \cite{L-Y-P}, in particular Corollary 11, where the assumption that {\em all} the weights of ${\cal T}$ are positive  is not necessary). 
Then $ {\cal T}'$ is a pseudostar of kind $(n,k)$.
\end{prop}

\begin{proof} Suppose, contrary to our claim, that there exists 
 an edge $e$ of $T'$ dividing $L(T')=[n]$ into two parts both of cardinality less than $k$. By Theorem \ref{LYP}, or more precisely 
 by the analogous statement for ip-weighted trees, 
the quartet system of $T'$ is contained in the quartet system of $T$,
so $T'$ is obtained from $T$ by contracting some edges (see Theorem 1 in \cite{BS}); thus $e$
     corresponds to an edge of  $T$ dividing $L(T)=[n]$ into two parts both of cardinality less than $k$. We can suppose $e$ is $\gamma_{a,b,c,d}$ for some
$a,b,c,d \in [n]$ such that $\langle a,b\,|\, c,d \rangle $ holds. Denote $s_{a,b}$ by $t$ 
ad $ s_{c,d}$ by $s$ (see Definition \ref{cherries}).
We want to show that 
\begin{equation} \label{eqbo}
  S_{a,b} + S_{c,d} = S_{a,c} + S_{b,d}
\end{equation}
(which is absurd since it implies that the weight of $e$ is equal to $0$).
  Obviously $S_{a,b}$ is equal to 
$$  \sum_{E \in {[n]-\{a,b,c,d\} \choose k-2}} D_{a,b,E}({\cal T}) + \sum_{E \in {[n]-\{a,b,c,d\} \choose k-3}} D_{a,b,c,E}({\cal T}) + \sum_{E \in {[n]-\{a,b,c,d\} \choose k-3}} D_{a,b,d,E}({\cal T}) +\sum_{E \in {[n]-\{a,b,c,d\} \choose k-4}} D_{a,b,c,d,E}({\cal T})$$
and analogously $S_{c,d}$, $S_{a,c}$ and $S_{b,d}$. Hence 
(\ref{eqbo}) is equivalent to 
\begin{equation} \label{eqbo2} 
 \sum_{E \in {[n]-\{a,b,c,d\} \choose k-2}} (D_{a,b,E}({\cal T}) 
+ D_{c,d,E}({\cal T}))= 
 \sum_{E \in {[n]-\{a,b,c,d\} \choose k-2}} (D_{a,c,E}({\cal T}) 
+ D_{b,d,E}({\cal T})).
\end{equation}
We can write  $E \in {[n]-\{a,b,c,d\} \choose k-2}$ as disjoint union of 
$E_a, E_b, E_c, E_d , E_t, E_s $, where: 

$E_{a} = \{x \in E |   \;x \; \mbox{\rm   clings to } p(a, t)-  \{t\} \; \mbox{\rm as to }    T|_{a,b,c,d}\}$

and analogously $E_b$, $E_c$, $E_d$,

$E_t = \{x \in E | \; x \; \mbox{\rm  clings to }  t \; \mbox{\rm as to }    T|_{a,b,c,d}\}$

and analogously $E_s$. By our assumption that $e$ divides $L(T)=[n]$ into two parts both of cardinality less than $k$, we have that  $E_a \cup E_b \cup
E_t \neq  \emptyset $ and   $E_c \cup E_d \cup E_s \neq  \emptyset $, in fact: define $A = E_a \cup E_b \cup E_t \cup \{a,b\}$ and 
$B=E_c \cup E_d \cup E_s \cup \{c,d\}$; we have that $E \cup \{a,b,c,d\}$ is the (disjoint) union of  
$A$ and $ B$, hence $ \#(A \cup B) = \#(E \cup \{a,b,c,d\})=k+2$; moreover $\#A \leq k-1$, $\#B \leq k-1$, therefore
   $\#A \geq 3 $ and $\#B \geq 3$, which gives the desired conclusion.

So  we get:
$$ D_{a,b,E}({\cal T}) = w(p(a,t)) +  w(p(b,t)) +w(e) +w(p(s, \overline{E_c}))
+w(p(s, \overline{E_d}))$$ $$ 
+ w\big(T|_{E_a ,t} -p(a,b)\big) + w\big(T|_{E_b ,t} -p(a,b)\big) + w\big(T|_{E_c ,s} -p(c,d)\big) +
 w\big(T|_{E_d ,s} -p(c,d)\big) , $$
where $\overline{E_c}$ is the vertex of $T|_{E_c ,s} \cap p(s,c)$ which is the most far from $s$ and analogously  $\overline{E_d}$. 
We can write $ D_{a,c,E}({\cal T})$,  $ D_{b,d,E}({\cal T})$,
$ D_{c,d,E}({\cal T})$ in an analogous way and we get that 
 $$D_{a,b,E}({\cal T}) + D_{c,d,E}({\cal T})= 
D_{a,c,E}({\cal T}) + D_{b,d,E}({\cal T}).$$
So (\ref{eqbo2}) holds.
\end{proof}

\begin{rem} \label{lemmaLYP}
 Let ${\cal T}$ be an ip-weighted tree with $L(T)=[n]$. 
Let $S_{i,j}$ for $i,j \in [n]$ be defined from 
the $ D_I ({\cal T})$ for  $I \in {[n ]\choose k}$
as in Definition \ref{Sij}.
Let ${\cal T}'=(T',w')$ be an essential ip-weighted  tree
 with $\{S_{i,j}\}$ as family of the $2$-weights (the existence follows from
 Theorem \ref{Bune2} and  paper \cite{L-Y-P}, in particular Corollary 11).
 It is a pseudostar of kind $(n,k)$ by Proposition \ref{T'pseudo} (so it is 
 equal to $T'_{\leq n-k} \cong T_{\leq n-k}$).
Let $e $ be an internal edge of $T'$ and let $a,b,c,d \in [n]$ such that
$\langle a,b\, |\, c,d \rangle $ holds and the bridge of $(a,b,c,d) $ is given only by the edge $e$; in \cite{L-Y-P} (in particular see Lemma 12), the authors proved that  
$$w(e)=  \frac{2w'(e)  }{ {\# L_{a,b}-2 \choose k-2}+ {\# L_{c,d}-2 \choose k-2}}$$
\end{rem}

\begin{thm} \label{ip-l}
Let $\{D_{I}\}_{I \in {[n] \choose k}}$ be a family in $\R$.
Let $S_{i,j}$ be defined from the family  $\{D_{I}\}_{I \in {[n] \choose k}}$ 
as in Definition \ref{Sij} and see Definition 
  \ref{defL} for the definition of $L_{a,b}$.
For any $W \in {[n] \choose k}$, let us denote
$$Q(W)= \left\{(a,b,c,d) \in {W\choose 4} \;| \;  \;\;
S_{a,b} + S_{c,d}
< S_{a,c} + S_{b,d}= S_{a,d} + S_{b,c}, \;\;\;L_{a,b} \cup L_{c,d}=[n] \right\}/ \sim$$ where 
$(a,b,c,d) \sim  (a',b',c',d')$ if and only if, up to swapping 
$\{a,b\}$ with $\{c,d\}$,  we have that $$\{a,b\} \subset L_{a',b'}, 
 \;\; \{c,d\} \subset L_{c',d'},\;\;\{a',b'\} \subset L_{a,b} , \;\;\{c',d'\} \subset L_{c,d}$$
 The family $\{D_{I}\}_{I \in {[n] \choose k}}$
 is ip-l-treelike if  and only if the following conditions hold:

(i) the $S_{i,j}$ for distinct $i,j \in [n]$ satisfy the   $4$-point condition;

(ii)  for any distinct $ a,b,c,d \in [n]$ such that 
$  S_{a,c} + S_{b,d} = S_{a,d} + S_{b,c}$, $L_{a,b} \cup L_{c,d}= [n]$ and 
$\# L_{a,b} < k$, $ \#L_{c,d} < k$, we have:  
\begin{equation} \label{eqabcd}
  S_{a,b} + S_{c,d} =  S_{a,c} + S_{b,d} \, ; \end{equation}

(iii)
 for any $I \in {[n] \choose k}$, we have: 
 $$ \sum_{i,j \in I} \left[ \frac{  D_{j,X} - D_{i,X}}{2}   + \left(\sum_{[(a,b,c,d) ]\in Q(iX) }
  -  \sum_{[(a,b,c,d)] \in Q(jX) } \right)
\frac{S_{a,c} + S_{b,d} -  S_{a,b} - S_{c,d}  }{ {\# L_{a,b}-2 \choose k-2}+ {\# L_{c,d}-2 \choose k-2}}
 \right]=0,$$
where
 $X = X_{i,j} $ is any element  of $ { {[n] -\{i,j\}\choose k-1 }}$ (so it depends on $i$ and $j$).

(iv) for any $i \in [n]$, 
 $$ D_I+ \sum_{j \in I} \left( 
D_{j,X} - D_{i,X} \right)+ \left[ \sum_{j \in I}  
 \left(\sum_{[(a,b,c,d) ]\in Q(iX) }  - \sum_{[(a,b,c,d) ]\in Q(jX) }\right) 
 -
\sum_{[(a,b,c,d) ]\in Q(I) }\right]
\frac{S_{a,c} + S_{b,d} -  S_{a,b} - S_{c,d}  }{ {\# L_{a,b}-2 \choose k-2}+ {\# L_{c,d}-2 \choose k-2}}
$$
    does not depend on $ I \in {[n] \choose k}$  and $X=X_{i,j} \in  {[n] -\{i,j\}\choose k-1}$.

\end{thm}

\begin{proof}
$\Longrightarrow$ Let ${\cal T}=(T,w)$ be  an ip-weighted tree with $L(T)=[n]$ and  realizing the family.
Condition (i) has been proved in \cite{L-Y-P}. Condition
(ii)  follows from Propositions \ref{bridge} and  \ref{T'pseudo} and (i), in fact: 
let $ a,b,c,d $ be distinct elements of $[n]$  such that 
$  S_{a,c} + S_{b,d} = S_{a,d} + S_{b,c}$, $L_{a,b} \cup L_{c,d}= [n]$, 
$\# L_{a,b} < k$ and $ \#L_{c,d} < k$; by condition (i), we have that 
$  S_{a,b} + S_{c,d} \leq   S_{a,c} + S_{b,d} $; let ${\cal T}' =( T',w') $ be  
the ip-weighted essential tree  with $L(T') =[n]$ and such that the $2$-weights are equal to the $S_{i,j}$; 
if, contrary to our claim, we had $  S_{a,b} + S_{c,d} <  S_{a,c} + S_{b,d} $,
then, by Proposition \ref{bridge},  $\langle a,b \, |\, c,d \rangle $ would hold and the bridge of $(a,b,c,d) $ would be given by exactly one edge;
since $\# L_{a,b} < k$ and $ \#L_{c,d} < k$, we would have that  ${\cal T}' $  is not a pseudostar of kind $(n,k)$, but this is absurd by 
Proposition \ref{T'pseudo}.

Let us prove  (iii). We have: $$ D_I({\cal T})=  
   \sum_{i \in I} w(e_i) + \sum_{e \in \mathring{E} (T|_I) } w(e)
=   $$  $$=
  \sum_{i \in I} \left(
  \frac{D_I({\cal T)}}{k} -\frac{1}{k}
\sum_{e \in \mathring{E} (T|_I) } w(e) + \frac{1}{k} \sum_{j \in I} \left( 
D_{j,X} ({\cal T})- D_{i,X} ({\cal T}) - \sum_{e \in \mathring{E} (T|_{jX})}
w(e) +\sum_{e \in \mathring{E} (T|_{iX})} w(e)
\right) \right) +$$ 
$$+ \sum_{e \in \mathring{E} (T|_I) } w(e)=$$
$$=  D_I({\cal T)}-
\sum_{e \in \mathring{E} (T|_I) } w(e) + \frac{1}{k}  \sum_{i,j \in I} \left( 
D_{j,X} ({\cal T})- D_{i,X} ({\cal T}) - \sum_{e \in \mathring{E} (T|_{jX})}
w(e) +\sum_{e \in \mathring{E} (T|_{iX})} w(e)
\right)  + \sum_{e \in \mathring{E} (T|_I) } w(e),$$
  where 
 $X = X_{i,j} $ is any element  of $ { {[n] -\{i,j\}\choose k-2 }}$ (so it depends on $i$ and $j$) and
  the second equality holds by Remark \ref{wei}. Hence  
  \begin{equation} \label{terza} \sum_{i,j \in I} \left( 
D_{j,X} ({\cal T})- D_{i,X} ({\cal T}) - \sum_{e \in \mathring{E} (T|_{jX})}
w(e) +\sum_{e \in \mathring{E} (T|_{iX})} w(e)
\right) =0. \end{equation}
 By Proposition \ref{bridge}, we have that, for any distinct $a,b,c,d \in [n]$, we have that 
  $\langle a,b \,|\, c,d \rangle $ holds and 
 the bridge of $(a,b,c,d)$ is given by exactly one edge if and only if $
S_{a,b} + S_{c,d} < S_{a,c} + S_{b,d}= S_{a,d} + S_{b,c}$ and $ 
L_{a,b} \cup L_{c,d}= [n]$. Moreover, given $a,b,c,d,a',b',c',d'$ such that 
  $\langle a,b \,|\, c,d \rangle $ holds and 
 the bridge of $(a,b,c,d)$ is given by exactly one edge and analogously 
 for $a',b',c',d'$,  we have that
$(a,b,c,d) $ and $(a',b',c',d')$
    give the same edge if and only if they are equivalent. So,
    for any $W \in {[n] \choose k}$, an internal edge of ${\cal T}|_W$ corresponds to an element of $Q(W)$. Hence, 
      from (\ref{terza}) and Lemma 12 in \cite{L-Y-P} (see Remark \ref{lemmaLYP}), we get   condition (iii). Finally, by Remark \ref{wei}, the fact that an internal edge of ${\cal T}|_W$ corresponds to an element of $Q(W)$ and  Lemma 12 in \cite{L-Y-P} (see Remark \ref{lemmaLYP}),  we get (iv).
    
$\Longleftarrow$ Let ${\cal T}' =(T',w')$ be an essential ip-weighted tree with    $2$-weights equal to the $S_{i,j}$ 
(it exists by condition (i) and Theorem \ref{Bune2}). It is a pseudostar of kind $(n,k)$ by condition (ii), in fact: let $e$ be an internal edge of  $ T'$; let $a,b,c,d \in [n]$ be such that  $\langle a,b, \, |\, c,d \rangle$
holds and the bridge of $(a,b,c,d)$ is given only by $e$; then 
 $  S_{a,b} + S_{c,d} <  S_{a,c} + S_{b,d} = S_{a,d} + S_{b,c}$ and 
$L_{a,b} \cup L_{c,d}= [n]$; if, contrary to our claim, we had 
$\# L_{a,b} < k$, $ \#L_{c,d} < k$, then by (ii), we would get a contradiction.

  Let ${\cal T}
=(T,w)$ be the weighted tree with $T=T'$ and where $w$ is defined as follows: for any $e \in  \mathring{E} (T') $, let $a,b,c,d \in [n] $ be such that $\langle a,b \,|\, c,d \rangle $ holds and the bridge of $(a,b,c,d) $ is $e$; define $$w(e)= \frac{2w'(e)}{{\# L_{a,b}-2 \choose k-2}+ {\#L_{c,d}-2 \choose k-2}} ;$$
hence  $$w(e)= 
\frac{S_{a,c} + S_{b,d} -  S_{a,b} - S_{c,d}  }{ {\# L_{a,b}-2 \choose k-2}+ {\# L_{c,d}-2 \choose k-2}};$$
moreover, for any $i \in [n]$, define 
$$ w(e_i) = \frac{D_I}{k} -\frac{1}{k}
\sum_{e \in \mathring{E} (T|_I) } w(e) + \frac{1}{k} \sum_{j \in I} \left( 
D_{j,X} - D_{i,X}  - \sum_{e \in \mathring{E} (T|_{jX})}
w(e) +\sum_{e \in \mathring{E} (T|_{iX})} w(e)
\right)$$
for any $ I \in {[n] \choose k}$  and any $X=X_{i,j} \in {[n] -\{i,j\}\choose k-1}$ (so it depends on $i$ and $j$). Observe that it is a good definition by condition (iv).
We have to show that $D_{I}({\cal T})= D_I$ for any $I  \in {[n] \choose k}$. We have: 
  $$ D_I({\cal T})=  
   \sum_{i \in I} w(e_i) + \sum_{e \in \mathring{E} (T|_I) } w(e)
=   $$  $$=  \sum_{i \in I} \left(
  \frac{D_I}{k} -\frac{1}{k}
\sum_{e \in \mathring{E} (T|_I) } w(e) + \frac{1}{k} \sum_{j \in I} \left( 
D_{j,X} - D_{i,X}  - \sum_{e \in \mathring{E} (T|_{jX})}
w(e) +\sum_{e \in \mathring{E} (T|_{iX})} w(e)
\right) \right) +$$ 
$$+ \sum_{e \in \mathring{E} (T|_I) } w(e)=$$
$$=  D_I-
\sum_{e \in \mathring{E} (T|_I) } w(e) + \frac{1}{k}  \sum_{i,j \in I} \left( 
D_{j,X} - D_{i,X}  - \sum_{e \in \mathring{E} (T|_{jX})}
w(e) +\sum_{e \in \mathring{E} (T|_{iX})} w(e)
\right)  + \sum_{e \in \mathring{E} (T|_I) } w(e),$$
  where 
 $X = X_{i,j} $ is any element  of $ { {[n] -\{i,j\}\choose k-2 }}$ (so it depends on $i$ and $j$) and
  the second equality holds by the definition of $w(e_i)$.
  So  $D_{I} ({\cal T})= D_I$ if and only if 
 $$  \sum_{i,j \in I} \left( 
D_{j,X} - D_{i,X}  - \sum_{e \in \mathring{E} (T|_{jX})}
w(e) +\sum_{e \in \mathring{E} (T|_{iX})} w(e)
\right)  =0,$$
 which is true by the definition of the weight of the internal edges and  by assumption (iii).
\end{proof}

\begin{rem} It is easy to get from Theorem \ref{ip-l} a characterization
also for p-l-treelike families.
Obviously a family  $\{D_{I}\}_{I \in {[n] \choose k}}$ is p-l-treelike if and only if conditions (i), (ii), (iii), (iv) of Theorem \ref{ip-l} hold and, in addition, the number displayed in (iv) is positive for any $i \in [n]$.

\end{rem}

{\bf Acknowledgements.}
This work was supported by the National Group for Algebraic and Geometric Structures and their  Applications (GNSAGA-INdAM). 
The first author was supported by Ente Cassa di Risparmio di Firenze.

\bigskip

{\bf Address:}
Dipartimento di Matematica e Informatica ``U. Dini'',

viale Morgagni 67/A,
50134  Firenze, Italia

\end{document}